\documentclass[a4paper,10pt]{article}

\usepackage[utf8]{inputenc}
\usepackage[english]{babel}
\usepackage[T1]{fontenc}
\usepackage{amsmath} 
\usepackage{amsfonts}
\usepackage{amssymb}
\usepackage{amsthm}
\usepackage{eufrak}
\usepackage{bbm}
\usepackage{mathtools}
\usepackage{hyperref}

\newcommand{\R}{{\mathbb{R}}}

\newcommand{\N}{{\mathbb{N}}}

\DeclareMathOperator{\argmin}{argmin}

\renewcommand{\epsilon}{\varepsilon}

\renewcommand\O[1]{\mathcal O_{#1}}
\newcommand{\energy}{{\mathcal{E}}}

\theoremstyle{plain}
\newtheorem{thm}{Theorem}[section]
\newtheorem{lem}[thm]{Lemma}
\newtheorem{prop}[thm]{Proposition} 

\theoremstyle{definition} 
\newtheorem{defi}{Definition}

\newtheorem{rmk}{Remark}

\title{A dyadic model on a tree}
\author{David Barbato \and Luigi Amedeo Bianchi \and Franco Flandoli \and Francesco Morandin}
\date{\today}

\begin{document}
\maketitle
\section{Introduction}
\label{sec:the_model}
A classical scheme used to explain energy cascade in turbulence, see
e.g.~\cite{K41} and~\cite{MR1428905}, is based on the picture of the
fluid as composed of eddies of various sizes. Larger eddies split into
smaller ones because of dynamical instabilities and transfer their
kinetic energy from their scale to the one of the smaller eddies. One
can think of a tree-like structure where nodes are eddies; any
substructure father-offsprings, where we denote the father by $j\in J$
($J$ the set of nodes) and the set of offspring by $\O j$, corresponds
to an eddy $j$ and the set $\O j$ of smaller eddies produced by $j$ by
instability. In the simplest possible picture, eddies belong to
specified discrete levels, \emph{generations}: level 0 is made of the
largest eddy, level 1 of the eddies produced by level zero, and so
on. The generation of eddy $j$ may be denoted by $|j|$. Denote also
the father of eddy $j$ by $\bar{\jmath}$.

Phenomenologically, we associate to any eddy $j$ a non-negative
\emph{intensity} $X_{j}(t) $, at time~$t$, such that the
kinetic energy of eddy $j$ is $X_{j}^{2}(t) $. We relate
intensities by a differential rule, which prescribes that the
intensity of eddy $j$ increases because of a flux of energy from
$\bar{\jmath}$ to $j$ and decreases because of a flux of energy
from $j$ to its set of offspring $\O j$. We choose the
rule%
\begin{equation}
\frac{d}{dt}X_{j}=c_{j}X_{\bar{\jmath}}^{2}-\sum_{k\in\O j}%
c_{k}X_{j}X_{k}\label{tree-like_dyadic}%
\end{equation}
where the coefficients $c_{j}$ are positive.

This model has been introduced by Katz and Pavlovi\'c~\cite{MR2095627}
as a simplified wavelet description of Euler equations, suitable for
understanding the energy cascade. The coefficients $c_j=2^{\alpha|j|}$
represent in our model the speed of the energy flow from an eddy to
its children. The coefficient $\alpha$ is an approximation, averaged
in time and space, of the rate of this speed. Regarding solutions of
Euler equations in dimension 3, it may happen (usually as a short term
phenomenon) that this speed is higher or lower, sometimes that the
process itself is reversed, that is the energy flows from the smaller
eddies to the bigger ones: this is known as
intermittency. In~\cite{MR2522972} and~\cite{CheFriPav2007} it is
shown using Bernstein's inequality that the rate~$\beta$ for the
dyadic 3D Euler model lies in the interval~$[1,\frac52]$ which
corresponds to~$\alpha\in[\frac52,4]$ for the tree dyadic model.  As
explained in section~\ref{subsection_K41}, the order of magnitude of
$c_j$ that correspond to K41 is:
\begin{equation}\label{cj_from_KP}
c_j\sim2^{\frac52|j|}.
\end{equation}

The tree dyadic model~\eqref{tree-like_dyadic} is a more structured
version of the so called dyadic model of turbulence. The latter is
based on variables $Y_{n}$ which represent a cumulative intensity of
shell $n$ (shell in Fourier or wavelet space) $n=0,1,2,\ldots$ Here,
on the contrary, shell $n$ is described by a set of variables, all
$X_{j}$'s with $\left\vert j\right\vert =n$, the different intensities
of eddies of generation $n$. The equations for $Y_{n}$ have the form%
\begin{equation}
\frac{d}{dt}Y_{n}=k_nY_{n-1}^{2}-k_{n+1}%
Y_{n}Y_{n+1}.\label{simple_dyadic}%
\end{equation}
Model~\eqref{tree-like_dyadic} is thus a little bit more realistic
than~\eqref{simple_dyadic}, although it is still extremely idealized
with respect to the true Fourier description of Euler equations.

All these models are formally conservative: the global kinetic energy
$\energy(t)=\sum_jX_j^2(t)$, or $\energy(t)=\sum_nY_n^2(t)$ depending
on the case, is formally constant in time; it can be easily seen in
both cases, using the telescoping structure of the series
$\dfrac{d\energy(t)}{dt}$.  However, in previous papers
(\cite{CheFriPav2007},~\cite{MR2746670}) it has been shown that the
dyadic model~\eqref{simple_dyadic} is not rigorously
conservative:\ anomalous dissipation occurs. The flux of energy to
high values of $n$ becomes so fast after some time of evolution that,
in finite time, part of the energy escapes to infinity in $n$.

The same question for the tree dyadic model~\eqref{tree-like_dyadic}
is more difficult. Intuitively, it is not clear what to expect. Even
if the global flux from a generation to the next one behaves similarly
to the shell case~\eqref{simple_dyadic}, energy may split between
eddies of the same generation, which increase exponentially in
number. Hence there is a lot of ``space'' (a lot of eddies) to
accommodate the large amount of energy which comes from progenitors in
the cascade.

The main result of this paper, Theorem~\ref{thm:main_1}, is the proof
of anomalous dissipation also for model~\eqref{tree-like_dyadic}. To
be precise, we have dissipation for a class of coefficients $c_{j}$
which covers~\eqref{cj_from_KP}. The proof is similar to the one
in~\cite{MR2746670} but requires new ideas and ingredients.

Apart from anomalous dissipation, we consider also stationary
solutions, showing the existence and uniqueness of such solutions in
Theorems~\ref{thm:main_2} and~\ref{thm:main_3}. This kind of argument
allows and requires a more general model to be studied, namely, one
needs to insert a forcing term (to find nontrivial stationary
solutions) and we are able to treat also the viscous analogous of the
tree dyadic model, adding the viscosity term $-\nu2^{\gamma|j|} X_j$ to
equation~\eqref{tree-like_dyadic}.  The most general model that we
introduce is thus system~\eqref{eq:system_tree_dyadic_general}.

In Section~\ref{sec:main_results} we describe the model and give a
short summary of the main results of the paper.

In Section~\ref{sec:elem-prop} we discuss elementary properties of the
model and prove the existence of finite energy solutions.

In Section~\ref{sec:dyadic-model-tree-vs-classical} we exploit the
connection between the ``classic'' dyadic model on naturals and the
tree dyadic model. If the number of children is constant for every
node in the tree, then from each solution on the former one can build
a ``lifted'' version on the tree which is a solution of the latter.

Section~\ref{sec:anom-diss} is devoted to the proof of the anomalous
dissipation Theorem~\ref{thm:main_1} in the inviscid unforced
case. Self-similar solutions are also discussed.

In Section~\ref{sec:stat-sol} we study the stationary solutions. We
prove existence and uniqueness of stationary solutions of classic and
tree forced systems~\eqref{eq:system_classic_dyadic_general}
and~\eqref{eq:system_tree_dyadic_general} with and without viscosity.
Here the positive force $f$ is required because otherwise the unique
non-negative stationary solution is the null one.

\subsection{The decay of $X_{j}$ corresponding to K41 and anomalous
dissipation\label{subsection_K41}}

In the case of the classic dyadic model~\eqref{simple_dyadic},
Kolmogorov inertial range spectrum reads
\[
Y_{n}\sim k_{n}^{-1/3}.
\]
The exponent is intuitive in such case. For the tree dyadic
model~\eqref{tree-like_dyadic} the correct exponent may look
unfamiliar and thus we give a heuristic derivation of it. The result
is that Kolmogorov inertial range spectrum corresponds to
\begin{equation}
X_{j}\sim2^{-\frac{11}{6}|j|}.\label{K41}%
\end{equation}

K41 theory \cite{K41} states that, if $u\left(  x\right)  $ is the velocity of
the turbulent fluid at position $x$ and the expected value $E$ is suitably
understood (for instance if we analyze a time-stationary regime), one has%
\[
E\bigl[|u(x)-u(y)|^2\bigr]\sim|x-y|^{2/3}%
\]
when $x$ and $y$ are very close each other (but not too close). Very
vaguely this means
\[
|u(x)-u(y)|\sim|x-y|^{1/3}.
\]
Following Katz-Pavlovi\'c \cite{MR2095627}, let us think that $u(x)$
may be written in a basis~$(w_j)$ (which are norm-one vectors in
$L^2$) as
\[
u(x)=\sum_{j}X_{j}w_{j}(x).
\]
The vector field $w_{j}(x)$ corresponds to the velocity field of eddy
$j$. Let us assume that eddy $j$ has a support $Q_{j}$ of the order of
a cube of side $2^{-|j|}$. Given $j$, take $x,y\in Q_{j}$. When we
compute $u(x)-u(y)$ we use the approximation $u(x)=X_{j}w_{j}(x)$,
$u(y)=X_{j}w_{j}(y)$. Then
\[
|u(x)-u(y)|=|X_j|\,|w_{j}(x)-w_{j}(y)|
\]
namely
\[
|X_{j}|\,|w_{j}(x)-w_{j}(y)|\sim|x-y|^{1/3},\qquad x,y\in Q_{j}.
\]
We consider reasonably correct this approximation when $x,y\in Q_{j}$ have a
distance of the order of $2^{-\left\vert j\right\vert }$, otherwise we should
use smaller eddies in this approximation. Thus we have%
\begin{equation}\label{passo_preliminare}
|X_{j}|\, |w_{j}(x)-w_{j}(y)| \sim2^{-\frac{1}{3}|j| },\qquad
x,y\in Q_{j},|x-y| \sim2^{-|j|}.
\end{equation}
Moreover, we have
\begin{equation}\label{mean_value}
|w_{j}(x)-w_{j}(y)|=|\nabla w_{j}(\xi)|\,|x-y| 
\end{equation}
for some point $\xi$ between $x$ and $y$ (to be precise, the mean
value theorem must be applied to each component of the vector valued
function $w_{j}$). Recall that $\int w_{j}(x)^{2}dx=1$, hence the
typical size $s_{j}$ of $w_{j}$ in $Q_{j}$ can be guessed from
$s_{j}^{2}2^{-3|j|}\sim1$, namely $s_{j}\sim2^{\frac32|j|}$. Since
$w_{j}$ has variations of order $s_{j}$ at distance $2^{-|j|}$, we
deduce that the typical values of $\nabla w_{j}$ in $Q_{j}$ have the
order $2^{\frac32|j|}/2^{-|j|}=2^{\frac52|j|}$. Thus, from~\eqref{mean_value},
\[
|w_{j}(x)-w_{j}(y)|\sim2^{\frac{5}{2}|j|}2^{-|j|}.
\]
Along with \eqref{passo_preliminare} this gives us
\[
|X_{j}|2^{\frac{5}{2}|j|}2^{-|j|}\sim2^{-\frac{1}{3}|j|}
\]
namely
\[
|X_{j}|\sim2^{(-\frac{1}{3}+1-\frac{5}{2})|j|}=2^{-\frac{11}{6}|j|}.
\]
We have established~\eqref{K41}, on a heuristic ground of course.

Let us give a heuristic explanation of the fact that, when anomalous
dissipation occurs, the decay~\eqref{K41} appears. In a sense, this
may be seen as a confirmation that~\eqref{K41} is the correct decay
corresponding to K41. Let us start from
equations~\eqref{tree-like_dyadic} with $c_{j}\sim2^{\frac{5}{2}|j|}$,
the Katz-Pavlovi\'c prescription.  Let $\energy_{n}$ be the energy up
to generation $n$:
\[
\energy_{n}=\sum_{|j|\leq n}X_{j}^{2}.
\]
Then, as will be seen later with equation~\eqref{eq:energy_derivative},
\[
\frac{d\energy_{n}}{dt}=-2^{\frac{5}{2}(n+1)}\sum_{|k|=n+1}X_{\bar k}^{2}X_{k}.
\]
In order to have anomalous dissipation, we should have
\[
\frac{d\energy_{n}}{dt}\overset{n}{\sim}-C\neq0.
\]
If we assume a power decay
\[
X_{j}\sim2^{-\eta|j| }.
\]
Then, since the cardinality of $\{Q_j:|j|=n\}$ should be of the order
of $2^{3n}$,
\[
2^{\frac{5}{2}(n+1)}\sum_{|k|=n+1}X_{\bar k}^{2}X_{k}
\sim2^{\frac{5}{2}n}2^{3n}2^{-3\eta n}
=2^{(\frac{11}{2}-3\eta)n}
\]
and thus $\eta=\frac{11}{6}$.

\section{Model and main results}\label{sec:main_results}
Let $J$ be the set of nodes. Inside $J$ we identify one special node,
called root or ancestor of the tree, which is denoted by 0. For all
$j\in J$ we define the generation number $|j|\in \N$ (such that
$|0|=0$), the set of offsprings of $j$, denoted by $\O j\subset J$,
such that $|k|=|j|+1$ for all $k\in \O j$ and a unique parent
$\bar\jmath$ with $j\in \O{\bar\jmath}$. The root 0 has no parent
inside $J$, but with slight notation abuse we will nevertheless use
the symbol $\bar0$ when needed.

For sake of simplicity we will suppose throughout the paper that the
cardinality of $\O j$ is constant, $\sharp \O j=:N_*$ for all $j\in
J$, but some results can be easily generalized at least to the case
where $\sharp\O j$ is positive and uniformly bounded.

It will turn out to be very important to compare $N_*$ to some
coefficients of the model. To this end we set also
$\tilde\alpha:=\frac12\log_2N_*$ so that $N_*=2^{2\tilde\alpha}$

The dynamics of the tree dyadic model is described by a family
$(X_j)_{j\in J}$ of functions $X_j:[0,\infty)\rightarrow\R$. Its
  general formulation is described by the equations below. (Notice
  that $X_{\bar 0}$ does not belong to the family and merely
  represents a convenient symbolic alias for the constant forcing
  term.)
\begin{equation}\label{eq:system_tree_dyadic_general}
\left\{\begin{aligned}
X_{\bar0}(t)&\equiv f \\
\frac d{dt} X_j 
&= -\nu d_j X_j + c_j X_{\bar\jmath}^2-\sum_{k\in\O j}c_{k}X_jX_k,
      &\forall j&\in J
\end{aligned}\right.
\end{equation}
Here we suppose that $f\geq0$, $\nu\geq 0$, and that the other
coefficients have an exponential behavior, namely $c_j=2^{\alpha
  |j|}$, $d_j=2^{\gamma |j|}$ with $\alpha>0$ and $\gamma>0$.

If $f=0$ we call the system \emph{unforced}, if $\nu=0$ we call it
\emph{inviscid}.

This system will usually come with an initial condition which will be
denoted by $X^0=(X^0_j)_{j\in J}$. One natural space for $X(t)$ to
live is $l^2(J;\R)$, which we will simply denote by $l^2$, the setting
being understood. The $l^2$ norm will be simply denoted by
$\|\cdot\|$.

\begin{defi}
Given $X^0 \in \R^J$, we call componentwise solution of
system~\eqref{eq:system_tree_dyadic_general} with initial condition
$X^0$ any family $X = (X_j)_{j\in J}$ of continuously differentiable
functions $X_j:[0,\infty)\to\R$ such that $X(0)=X^0$ and all equations
  in system~\eqref{eq:system_tree_dyadic_general} are satisfied.  If
  moreover $X(t)\in l^2$ for all $t\geq0$, we call it an $l^2$
  solution.

We say that a solution is positive if $X_j(t)\geq0$ for all $j\in J$
and $t\geq 0$.
\end{defi}

Existence of positive $l^2$ solutions is classical and can be found in
Section~\ref{sec:elem-prop}, while uniqueness is an open problem.

This system of equations is locally conservative, in the sense made
rigorous by Proposition~\ref{prop:energy_balance} below, where the
following energy balance inequality is proven
\[
\|X(t)\|^2
\leq\|X(s)\|^2+2f^2\int_s^tX_0(u)du-2\nu\sum_{j\in J}d_j\int_s^tX_j^2(u)du
\]
It turns out that in some cases this is in fact an equality and in
some cases it is a strict inequality. When the latter happens we say
that anomalous dissipation occurs.

The main results of the paper deal with anomalous dissipation and
stationary solutions.

\subsubsection*{Anomalous dissipation on the inviscid, unforced tree dyadic model.}
The proof of the next result is given in Section~\ref{sec:anom-diss}.
\begin{thm}\label{thm:main_1}
Let $\sharp\O j=2^{2\tilde\alpha}$ for all $j$. Suppose
$\tilde\alpha<\alpha$ and $f=\nu=0$ in
equations~\eqref{eq:system_tree_dyadic_general}. Let $X$ be any
positive $l^2$ solution with initial condition $X^0$. Then there
exists $C>0$, depending only on $\|X^0\|$, such that for all $t>0$
\[
\energy(t):=\|X(t)\|^2:=\sum_{j\in J} X_j^2(t)<\frac{C}{t^2}\ .
\]
\end{thm}
This theorem holds also if we use the weaker hypothesis $1\leq\sharp
\O j\leq2^{2\tilde\alpha}$ for all $j$. The statement tells us that the
energy of the system goes to zero at least as fast as $t^{-2}$.  In
Section~\ref{subsec:self_similar} we show that for this model there
are some self-similar solutions and that their energy goes to zero
exactly like $t^{-2}$. So the estimate of Theorem~\ref{thm:main_1}
cannot be improved much.

\subsubsection*{Stationary solutions for the forced classic dyadic model.}
It will be important for our purposes to switch between the tree
dyadic model and the classic one, where $J$ is simply the set of
non-negative integers with $\O j:=\{j+1\}$ for all $j$.

To avoid confusion we will use different symbols for the classic
system, whose equations are the following.
\begin{equation}\label{eq:system_classic_dyadic_general}
\left\{\begin{aligned}
Y_{-1}(t) &\equiv f\\
\frac{d}{dt} Y_n &= -\nu l_n Y_n +
      k_n Y_{n-1}^2-k_{n+1}Y_nY_{n+1},
      &\forall n&\geq0
\end{aligned}\right.
\end{equation}
with $f\geq0$, $\nu\geq 0$, $k_n=2^{\beta n}$, $l_n=2^{\gamma n}$,
$\beta>0$ and $\gamma>0$.

When this model is interpreted as a special case
of~\eqref{eq:system_tree_dyadic_general} we will have $N_*=1$,
$\tilde\alpha=0$ and $\beta=\alpha$. Observe that the definitions of
solutions given on the tree model extend easily to this one, but
notice that in this setting $l^2$ will correspond to the standard
space of sequences.

The following theorem deals with \emph{stationary} solutions, namely
solutions constant in time. We do not detail the proof, since, by what
we said above, it is a special case of the analogous statement for the
tree dyadic model, Theorem~\ref{thm:main_3} which is proven in
Section~\ref{sec:stat-sol}.
\begin{thm}\label{thm:main_2}
If $f>0$, then there exists a unique $l^2$ positive solution $Y$ of
system \eqref{eq:system_classic_dyadic_general} which is stationary.
Moreover
\begin{description}
  \item{if $\nu=0$} then $Y_n(t):=f\cdot 2^{-\frac{\beta}{3}(n+1)}$; 
  \item{if $\nu>0$ and $3\gamma \geq 2\beta $}, the stationary
    solution is conservative and \emph{regular}, in that for all real
    $s$, $\sum_n [2^{sn}Y_n(t)]^2<\infty $;
  \item{if $\nu>0$ and $3\gamma < 2\beta $}, there exists $C>0$ such
    that for all $f>C$ the invariant solution of
    \eqref{eq:system_classic_dyadic_general} is not regular and
    exhibits anomalous dissipation.
\end{description}
\end{thm}
\noindent
In the inviscid case, this theorem extends an analogue result of
\cite{CheFriPav2007} where it is proved for $\beta=\frac{5}{2}$.  In
the viscous case it extends a result of \cite{MR2522972}, in which
existence and uniqueness of stationary solutions are proved for
$\gamma=2$ and $\beta\in\left(\frac{3}{2},\frac{5}{2} \right]$.

\subsubsection*{Stationary solutions for the forced tree dyadic model.}
An analogous of Theorem~\ref{thm:main_2} holds for the tree dyadic
model too. This is proved in Section~\ref{sec:stat-sol}.
\begin{thm}\label{thm:main_3}
Let $\sharp\O j=2^{2\tilde\alpha}$ for all $j$. Suppose
$\tilde\alpha<\alpha$ and $f>0$ in
equations~\eqref{eq:system_tree_dyadic_general}.  Then there exists a
unique $l^2$ positive solution $X$ which is \emph{stationary}.
Moreover
\begin{description}
  \item if $\nu=0$ then $X_j(t):=f\cdot
  2^{-\frac{|j|+1}{3}(2\tilde\alpha+\alpha)}$ for all $j\in J$;
  \item if $\nu>0$ and $0<\alpha-\tilde\alpha\leq\frac32\gamma$,
    the stationary solution is conservative and
    \emph{regular}, in that for all real $s$, $\sum_{j\in
      J}[2^{s|j|}X_j(t)]^2<\infty $;
  \item if $\nu>0$ and $\alpha-\tilde\alpha>\frac32\gamma$, there
    exists $C>0$ such that for all $f>C$ the invariant solution of
    \eqref{eq:system_tree_dyadic_general} is not regular and exhibits
    anomalous dissipation.
\end{description}
\end{thm}

\section{Elementary properties}
\label{sec:elem-prop}
We will provide, in this section, some basic results on the tree
dyadic model. The results are analogous to those provided for the
dyadic model in~\cite{MR2746670} and~\cite{MR2796837}, but the proofs
require some new ideas to cope with the more general structure.

We will suppose throughout the paper that the initial condition $X^0$
is in $l^2$ and that $X^0_j\geq0$ for all $j\in J$. It will turn out
that this two properties hold then for all times.

\begin{defi}
For $n\geq-1$, we denote by $\energy_n(t)$ the total energy on nodes
$j$ with $|j|\leq n$ at time $t$ and $\energy(t)$ the energy of all
nodes at time $t$ (which is possibly infinite):
\begin{align*}
\energy_n(t)&:=
\sum_{|j|\leq n}X_j^2(t), &
\energy(t)&:=
\sum_{j\in J}X_j^2(t).
\end{align*}
Note in particular that $\energy_{-1}\equiv0$.
\end{defi}

We will use very often the derivative of $\energy_n$, for $n\geq0$,
\begin{multline*}
\frac d{dt}\energy_n(t)
=2\sum_{|j|\leq n}X_j\frac d{dt}X_j(t)\\
=-2\nu\sum_{|j|\leq n}d_j X_j^2 + 2\sum_{|j|\leq n}c_j X_{\bar\jmath}^2X_j-2\sum_{|j|\leq n}\sum_{k\in\O j}c_{k}X_j^2X_k\\
=-2\nu\sum_{|j|\leq n}d_j X_j^2 + 2c_0 X_{\bar0}^2X_0-2\sum_{|k|= n+1}c_{k}X_{\bar k}^2X_k
\end{multline*}
so we get for all $n\geq0$
\begin{equation}\label{eq:energy_derivative}
\frac d{dt}\energy_n(t)
=-2\nu\sum_{|j|\leq n}d_j X_j^2(t) + 2f^2X_0(t)-2\sum_{|k|= n+1}c_{k}X_{\bar k}^2(t)X_k(t).
\end{equation}

\begin{prop}\label{prop:globalboundenergy}
If $X^0_j\geq0$ for all $j$, then any componentwise solution is
positive.  If $X^0$ is in $l^2$, any positive componentwise solution
is a positive $l^2$ solution, in particular for all $t\geq0$,
\begin{equation}\label{eq:globalboundenergy}
\energy(t)\leq(\energy(0)+1)e^{2f^2t}.
\end{equation}
\end{prop}

\begin{proof}
  From the definition of componentwise solution we get that for all $j\in J$
  \begin{equation}\label{eq:2}
   X_j(t)= \ X_j^0
      e^{-\int_0^t (\nu d_j+ \sum_{k} c_kX_k(r))dr}
    +\int_0^t c_jX^2_{\bar{\jmath}}(s)
      e^{-\int_s^t (\nu d_j+ \sum_{k} c_kX_k(r))dr}
      ds
  \end{equation}
yielding $X_j(t)\geq0$ for all $t>0$ and all $j\in J$. 

Now we turn to the estimates of $\energy(t)$.
In~\eqref{eq:energy_derivative}, since $X_k(t)\geq0$ we have two
negative contribution which we drop and we use the bound $X_0(t)\leq
X_0^2(t)+1\leq\energy_n(t)+1$ to get that for all $n\geq0$,
\[
\frac d{dt}\energy_n(t)
\leq 2f^2(\energy_n(t)+1)
\]
so by Gronwall lemma
$\energy_n(t)+1\leq(\energy_n(0)+1)e^{2f^2t}$. Letting $n\to\infty$ we
obtain~\eqref{eq:globalboundenergy}.
\end{proof}

\begin{prop}\label{prop:energy_balance}
For any positive $l^2$ solution $X$, the following energy balance
principle holds, for all $0\leq s<t$.
\begin{multline}\label{eq:conservativity}
\energy(t)
=\energy(s)+2f^2\int_s^tX_0(u)du-2\nu\sum_{j\in J}d_j\int_s^tX_j^2(u)du\\
-2\lim_{n\to\infty}\int_s^t\sum_{|k|=n}c_kX_{\bar k}^2(u)X_k(u)du
\end{multline}
where the limit always exists and is non-negative.
In particular, for the unforced, inviscid ($f=\nu=0$) tree dyadic
model, $\energy$ is non-increasing.
\end{prop}
\begin{proof}
Let $0\leq s<t$, then by~\eqref{eq:energy_derivative} for all $n\geq0$,
\begin{multline*}
\energy_n(t)
=\energy_n(s)-2\nu\sum_{|j|\leq n}d_j\int_s^tX_j^2(u)du + 2f^2\int_s^tX_0(u)du\\
-2\int_s^t\sum_{|k|=n+1}c_kX_{\bar k}^2(u)X_k(u)du
\end{multline*}
As $n\to\infty$, since the solution is in $l^2$,
$\energy_n(s)\uparrow\energy(s)<\infty$ and the same holds for
$t$. The viscosity term is a non-decreasing sequence bounded by
\[
2\nu\sum_{|j|\leq n}d_j\int_s^tX_j^2(u)du
\leq\energy(s) + 2f^2\int_s^tX_0(u)du<\infty
\]
so it converges too. Then the border term converges being the sum of
converging sequences.
\end{proof}

\begin{defi}\label{def:conservative}
We say that a positive $l^2$ solution $X$ is conservative in $[s,t]$
if the limit in~\eqref{eq:conservativity} is equal to zero that is if
\[
\energy(t)
=\energy(s)+2f^2\int_s^tX_0(u)du-2\nu\sum_{j\in J}d_j\int_s^tX_j^2(u)du
\]
Otherwise we say that $X$ has anomalous dissipation in $[s,t]$.
\end{defi}

\begin{thm}
Let $X^0\in l^2$ with $X_j^0\geq0$ for all $j\in J$. Then there exists
at least a positive $l^2$ solution with initial condition $X^0$.
\end{thm}

\begin{proof}
The proof by finite dimensional approximates is completely classic.
Fix $N\geq1$ and consider the finite dimensional system
\begin{equation}\label{eq:galerkin}
\begin{cases}
X_{\bar0}(t)\equiv f &\\
\frac d{dt}X_j=-\nu d_j X_j +c_j
      X_{\bar{\jmath}}^2-\sum_{k\in\O j}c_kX_jX_k
     & j\in J,\ 0\leq |j|\leq N\\
X_k(t)\equiv 0 & k\in J,\ |k|=N+1\\
X_j(0)=X^0_j &j\in J,\ 0\leq|j|\leq N
\end{cases},
\end{equation}
for all $t\geq0$. Notice that proposition~\ref{prop:globalboundenergy}
is true also for this truncated system (with unchanged proof), so
there is a unique global solution. (Local existence and uniqueness
follow from the local Lipschitz continuity of the vector field and
global existence comes from the bound
in~\eqref{eq:globalboundenergy}.)  We'll denote such unique solution
by $X^N$.

Now fix $j\in J$ and consider on a bounded interval $\left[0,T\right]$
the family $(X_j^N)_{N>|j|}$. By~\eqref{eq:globalboundenergy} we have
a strong bound that does not depend on $t$ and $N$
\begin{equation*}
   |X_j^N(t)|  \leq 
    (\energy(0)+1)^{\frac{1}{2}}
    e^{\frac{1}{2}Tf^2} \qquad \forall N\geq1\ \ \forall t\in [0,T]\ 
,
\end{equation*}
thus the family $(X_j^N)_{N>|j|}$ is uniformly bounded, and by
applying the same bound to \eqref{eq:galerkin}, equicontinuous.  From
Arzel\`a-Ascoli theorem, for every $j\in J$ there exists a sequence
$(N_{j,k})_{k\geq1}$ such that $(X_j^{N_{j,k}})_k$ converges uniformly
to a continuous function $X_j$. By a diagonal procedure we can modify
the extraction procedure and get a single sequence $(N_k)_{k\geq1}$
such that for all $j\in J$, $X_j^{N_k}\to X_j$ uniformly. Now
we can pass to the limit as $k\to\infty$ in the equation
\[
X_j^{N_k}=X_j^0+\int_0^t\Bigl[
-\nu d_jX_j^{N_k}(r)+
c_j\left(X^{N_k}_{\bar
      \jmath}(r)\right)^2-\sum_{i\in\O j}c_iX_j^{N_k} (r)X_i^{N_k} (r)\Bigr]dr
\]
and prove that the functions $X_j$ are continuously differentiable and
satisfy system~\eqref{eq:system_tree_dyadic_general} with initial
condition $X^0_j$. Continuation from an arbitrary bounded time interval
to all $t\geq0$ is obvious.  Finally, $X$ is a positive $l^2$ solution
by Proposition~\ref{prop:globalboundenergy}.
\end{proof}

We conclude the section on elementary results by collecting a useful
estimate on the energy transfer and a statement clarifying that all
components are strictly positive for $t>0$.

\begin{prop}\label{prop:energy_inequality}
  The following properties hold:
\begin{enumerate}
\item If $f=0$, for all $n\geq-1$ 
\begin{equation}\label{eq:flux_inequality}
  2\int_0^{+\infty} \sum_{|k|=n+1} c_k X_{\bar k}^2(s)X_k(s) ds
\leq \energy_n(0)
 \end{equation}
\item if $X_j^0>0$ for all $j$ s.t.~$|j|=M$ for some $M\geq0$, then
  $X_j(t)>0$ for every $j$ s.t.~$|j|\geq M$ and all $t>0$.
\end{enumerate}
\end{prop}

\begin{proof}
  \begin{enumerate}
  \item If $n=-1$ the inequality is trivially true. If $n\geq0$, by
    integrating equation~\eqref{eq:energy_derivative} with $f=0$, we
    find that
\[
\energy_n(t)+2\nu\int_0^t\sum_{|j|\leq n}d_jX_j^2(s)ds=\energy_n(0) -2\int_0^t\sum_{|k|= n+1}c_{k}X_{\bar k}^2(s)X_k(s)ds
\]
The left hand side is non-negative for all $t$, so taking the limit
for $t\to \infty$ in the right hand side completes the proof.
\item For $|j|=M$ we have from \eqref{eq:2}
  \begin{equation*}
    \label{eq:10}
    X_j(t)\geq X_j^0
    e^{-\int_0^t (\nu d_j+ \sum_{k} c_kX_k(r))dr}>0
  \end{equation*}
  Now suppose that for some $j\in J\setminus\{0\}$,
  $X_{\bar\jmath}(t)>0$ for every $t>0$. Then again by~\eqref{eq:2},
  \begin{equation*}
    \label{eq:11}
    X_j(t)
    \geq\int_0^t c_jX^2_{\bar\jmath}(s)
    e^{-\int_s^t (\nu d_j+ \sum_{k\in\O j} c_kX_k(r))dr}ds>0
  \end{equation*}
  By induction on $|j|\geq M$ we have our thesis.\qedhere
  \end{enumerate}
  
\end{proof}

\section{Relationship with classic dyadic model}
\label{sec:dyadic-model-tree-vs-classical}
Recall the differential equations for the tree and classic dyadic
models.
\begin{align}\label{eq:system_albero}
&\left\{\begin{aligned}
X_{\bar0}(t)&\equiv f \\
\frac d{dt} X_j 
&= -\nu d_j X_j + c_j X_{\bar\jmath}^2-\sum_{k\in\O j}c_{k}X_jX_k,
      &\forall j&\in J
\end{aligned}\right.\\
\label{eq:system_diadico}
&\left\{\begin{aligned}
Y_{-1}(t) &\equiv f\\
\frac{d}{dt} Y_n &= -\nu l_n Y_n +
      k_n Y_{n-1}^2-k_{n+1}Y_nY_{n+1},
      &\forall n&\geq0
\end{aligned}\right.
\end{align}
where $f\geq0$, $\nu\geq0$ and for all $n\in\N$ and $j\in J$,
\begin{equation*}
\begin{aligned}
c_j&=2^{\alpha|j|}, &  
k_n&=2^{\beta n}, &
d_j&=2^{\gamma |j|}, &
l_n&=2^{\gamma n}.
\end{aligned}
\end{equation*}
Again we assume that $\sharp\O j=N_*=2^{2\tilde\alpha}$ for all $j\in
J$, but we stress that for this section this is a fundamental
hypothesis and not a technical one.

The following proposition shows that examples of solutions of the
tree dyadic model~\eqref{eq:system_albero} can be obtained by lifting
the solutions of the classic dyadic model~\eqref{eq:system_diadico}.
\begin{prop}\label{prop:class_to_tree}
If $Y$ is a componentwise (resp.\ $l^2$) solution
of~\eqref{eq:system_diadico}, then
$X_j(t):=2^{-(|j|+2)\tilde\alpha}Y_{|j|}(t)$ is a componentwise
(resp.\ $l^2$) solution of~\eqref{eq:system_albero} with
$\alpha=\beta+\tilde\alpha$. If $Y$ is positive, so is $X$.
\end{prop}
\begin{proof}
A direct computation shows that $X$ is a componentwise solution. Then
observe that, for any $n\geq0$,
\[
\sum_{|j|=n}X_j^2
=2^{2\tilde\alpha n}X_j^2
=2^{2\tilde\alpha n}2^{-(2n+4)\tilde\alpha}Y_{n}^2
=2^{4\tilde\alpha}Y_n^2
\]
so
\[
\energy_n
=\sum_{|j|\leq n}X_j^2
=\sum_{k\leq n}2^{4\tilde\alpha}Y_k^2
\leq2^{4\tilde\alpha}\|Y\|^2
\]
Positivity is obvious.
\end{proof}

\begin{rmk}
If we consider $\alpha$ fixed, since $\beta=\alpha-\tilde\alpha$, for
small values of $N_*$ we'll have larger values of $\beta$, and the
other way around. That is to say, the less offspring every node has,
the faster the dynamics will be.
\end{rmk}


\begin{rmk}
Let us stress that $\beta >0$ when $N_*<2^{2\alpha}$. Since the
behavior of the solutions of \eqref{eq:system_diadico} is strongly
related to the sign of $\beta$, then the behavior of the solutions of
\eqref{eq:system_albero} is strongly connected to the sign of $\alpha
-\tilde\alpha$. For example, in the classic dyadic there is anomalous
dissipation if and only if $\beta>0$, and hence in the tree dyadic
there will be lifted solutions with anomalous dissipation when
$\alpha>\tilde\alpha$ and lifted solutions which are conservative when
$\alpha\leq\tilde\alpha$.
\end{rmk}

\section{Anomalous dissipation and self-similar solutions\\
in the inviscid and unforced case.}
\label{sec:anom-diss}
Throughout this section we'll consider
system~\eqref{eq:system_tree_dyadic_general} in its unforced ($f=0$)
and inviscid ($\nu=0$) version.

\begin{equation}\label{eq:System_inviscid_unforced}
\left\{\begin{aligned}
X_{\bar0}(t)&\equiv 0 \\
\frac d{dt} X_j 
&= c_j X_{\bar\jmath}^2-\sum_{k\in\O j}c_{k}X_jX_k,
      &\forall j&\in J
\end{aligned}\right.
\end{equation}

Equation~\eqref{eq:energy_derivative}, that is the derivative of
energy up to the $n$-th generation becomes
\[
\frac d{dt}\energy_n(t)=-2\sum_{|k|= n+1}c_{k}X_{\bar k}^2(t)X_k(t),\qquad n\geq0
\]
Since only the border term survives, one would expect it to vanish in
the limit $n\to\infty$.  This can be rigorously proven only if the
solution lives in a sufficiently regular space, that is to say that
$X_j^2$ goes fast to zero as $|j|\to\infty$. For the classic dyadic
Kiselev and Zlato\v{s}~\cite{KisZla} proved that solutions that are
regular in the beginning, stay regular for some time but then lose
regularity in finite time. Thus our analysis is not restricted to
regular solutions, and in fact we will prove in this section that for
sufficiently large times all solutions dissipate energy.

Let us give some definitions. Let us denote by $\gamma_j$ the energy
at time 0 in the subtree $T_j$ rooted in $j$ plus all the energy
flowing in $j$ from the upper generations,
\[
\gamma_j \coloneqq \sum_{k\in T_j}X_k^2(0) + \int_0^\infty
2c_jX_jX_{\bar{\jmath}}^2ds
\]
Let $0\leq s <t$ and define for all $j\in J$
\[
m_j\coloneqq \inf_{r\in\left[s,t\right]}X_j\left(r\right)
\]

\begin{lem} \label{lem:stime_sup_inf} 
Let $X$ be a positive $l^2$ solution of system
\eqref{eq:System_inviscid_unforced}. The following inequalities hold
for all $n\geq0$.
\begin{gather*}
\energy_n(t)-\energy_{n-1}(s)
\leq\sum_{|j|=n}m_j^2
\leq\energy(0) \\
\sum_{|j|=n}\gamma_j\leq \energy(0)\\
\sum_{k\in T_j}X_k(r)^2\leq\gamma_j,\qquad\forall r\geq0
\end{gather*}
\end{lem}

\begin{proof}
The upper bound is obvious, since
\[
\sum_{|j|=n}m_j^2
\leq\sum_{|j|=n}X_j(s)^2
\leq\energy_n(s)
\leq\energy(0)
\]
where we used Proposition~\ref{prop:energy_balance}. Now let $j\in
J$. From~\eqref{eq:System_inviscid_unforced} we have for the
differential of $X^2_j$
\begin{equation*}
\frac{d}{dt} X_j^2 = 2c_j
      X_{\bar{\jmath}}^2X_j-\sum_{k\in\O j}2c_kX_j^2X_k\ ,
\end{equation*}
Let $r\in[s,t]$ and integrate on $[s,r]$, yielding
\begin{equation*}
 X_j^2 (r)=X_j^2 (s) +\int_s^r 2c_j
      X_{\bar{\jmath}}^2(\tau)X_j(\tau) d\tau -\sum_{k\in\O j} \int_s^r 2c_kX_j^2(\tau)X_k(\tau) d\tau
\end{equation*}
Choosing now $r\in\argmin_{\left[s,t\right]}X_j$, we get
\[
m_j^2
\geq X_j^2(s)-\sum_{k\in\O j}\int_s^t2c_kX_{\bar k}^2(\tau)X_k(\tau)d\tau
\]
By summation over all nodes $j$ with $|j|=n$ we have
\[
\sum_{|j|=n}m_j^2
\geq \sum_{|j|=n}X_j^2(s)-\int_s^t\sum_{|k|=n+1}2c_kX_{\bar k}^2(\tau)X_k(\tau)d\tau.
\]
Finally, we apply for $m=n-1$, $n$ the following integral form
of~\eqref{eq:energy_derivative} to get the first part of the
thesis. (Even if $n=0$ and $m=-1$ this is true, trivially.)
\[
\energy_m(t)-\energy_m(s)
=-\int_s^t\sum_{|j|=m+1}2c_jX_{\bar\jmath}^2(\tau)X_j(\tau)d\tau
\]
We turn to the second part. Sum $\gamma_j$ on every~$j$ with $|j|=n$
to get
\[
\sum_{|j|=n}\gamma_j=\sum_{|k|\geq n}X_k^2\left(0\right)+\int_0^\infty 2\sum_{|j|=n}c_jX^2_{\bar{\jmath}}X_jds,
\]
by~\eqref{eq:flux_inequality} the integral term is bounded above by
$\energy_{n-1}(0)$, so
\[
\sum_{\left|j\right|=n}\gamma_j\leq
\sum_{|k|\geq n} X_k^2(0)
+\energy_{n-1}(0)=
\sum_{k\in J} X_k^2(0)=
 \energy(0).
\]
Finally, the third part. Let $r\geq0$. By computing the time
derivative of $\sum_{\substack{k\in T_j\\|k|\leq n}}X_k^2$ which is
analogous to~\eqref{eq:energy_derivative}, dropping the border term
and integrating on $[0,r]$, we have,
\[
\sum_{\substack{k\in T_j\\|k|\leq n}}X_k(r)^2
\leq\sum_{\substack{k\in T_j\\|k|\leq n}}X_k(0)^2+2\int_0^r2c_jX_jX_{\bar{\jmath}}^2du
\leq\gamma_j
\]
Now, let $n\to\infty$ to conclude.
\end{proof}

The following statement will be used in the proof of
Lemma~\ref{lem:en_dissip}.
\begin{lem}\label{lem:esp_inequality}
 For every $h>0$ and $\lambda>0$ the following inequality holds:
\[
\int_0^h \int_0^s e^{-\lambda(s-r)}dr\ ds\geq \frac{h}{2\lambda}\left(1-e^{-\lambda\frac{h}{2}}\right).
\]
\end{lem}
\begin{proof}
\[
\int_0^h \int_0^s e^{-\lambda(s-r)}dr\ ds\geq
\int_{\frac{h}{2}}^h \int_{s-\frac{h}{2}}^s e^{-\lambda(s-r)}dr\ ds=
\frac{h}{2\lambda}\left(1-e^{-\lambda\frac{h}{2}}\right).\qedhere
\]
\end{proof}

\begin{lem}
\label{lem:en_dissip}
  Assume that $\alpha>\tilde\alpha$, where $2^{2\tilde\alpha}=N_*=
  \sharp\O j$ is the constant number of children for every node. Let
  $X$ be a positive $l^2$ solution of
  \eqref{eq:System_inviscid_unforced}.  Let $(\delta_n)_{n\geq0}$ be a
  sequence of positive numbers such that $\sum_n\delta_n$ and
  $\sum_n\delta_n^{-2}2^{-(\alpha-\tilde\alpha) n}$ are both
  finite. Then there exists a sequence of positive numbers
  $(h_n)_{n\geq0}$ such that $\sum_nh_n<\infty$ and for all $n\geq0$ for
  all $t>0$
\begin{equation}
  \label{eq:tesi_lemma1}
  \energy_n(t+h_n)-\energy_{n-1}(t)\leq \delta_n.
\end{equation}
In particular, for every $M\geq0$,
\begin{equation}
\label{eq:tesi_lemma2}
\energy\biggl(\sum_{n=M}^\infty h_n\biggr)
\leq\energy_{M-1}(0)+\sum_{n=M}^\infty\delta_n.
\end{equation}

The sequence
\begin{equation}
  \label{eq:h_n_time}
h_n=\frac{\energy(0)^{3/2}}{\delta_n^2}2^{-(\alpha-\tilde\alpha)n+3/2},
\end{equation}
satisfies~\eqref{eq:tesi_lemma1} and~\eqref{eq:tesi_lemma2}.
\end{lem}

\begin{proof}
Fix $n\geq0$ and positive real numbers $t$, $h_n$. For all $j$ of
generation $n$, let $m_j\coloneqq \inf_{r\in[t,t+h_n]}X_j(r)$.  We
claim that if $h_n$ is defined by~\eqref{eq:h_n_time}, then $\sum_{|j|=n}
m_j^2 \leq \delta_n$, which together with
Lemma~\ref{lem:stime_sup_inf} completes the proof
of~\eqref{eq:tesi_lemma1}.

We prove the claim by contradiction: suppose that
$\sum_{|j|=n}m_j^2>\delta_n$. We will find a contradiction in the
estimates on $\energy(0)$. By Proposition~\ref{prop:energy_inequality}
\[
\energy(0)\geq 2\int_0^{h_n}\sum_{|j|=n}\sum_{k\in\O j}c_kX_k(t+s)X_j^2(t+s)ds
\]
We have a lower bound for $X_j$, namely $m_j$, but we need one also
for $X_k$.

For all $j\in J$, let
$\Gamma_j:=\max(\gamma_j,\energy(0)N_*^{-|j|})$. From
Lemma~\ref{lem:stime_sup_inf} we have
$\sum_{|j|=n}\gamma_j\leq\energy(0)$ and hence
$\sum_{|j|=n}\Gamma_j\leq2\energy(0)$; by the same lemma, for all
$i\in T_j$ we have $X_i^2\leq\gamma_j\leq \Gamma_j$ uniformly in time,
so for all $k\in\O j$,
\begin{equation*}
  \dot X_k=c_kX_j^2 - \sum_{i\in\O k}c_iX_iX_k
  \geq c_km_j^2 - \lambda_jX_k
\end{equation*}
where $\lambda_j=N_*2^{n\alpha+2\alpha}\sqrt{\Gamma_j}$. This gives
\begin{equation*}
  X_k(t+s)
  \geq c_km_j^2 \int_0^se^{-\lambda_j(s-r)}dr
\end{equation*}
We can write
\[
\energy(0)
\geq 2\sum_{|j|=n}m_j^4\int_0^{h_n} \int_0^se^{-\lambda_j(s-r)}drds\sum_{k\in\O j}c_k^2
\]
and by lemma~\ref{lem:esp_inequality} we have
\[
\energy(0)
\geq 2\sum_{|j|=n}m_j^4\frac{h_n}{2\lambda_j}\left(1-e^{-\lambda_jh_n/2}\right)\sum_{k\in\O j}c_k^2
\]
Let us focus on the exponential. We substitute~\eqref{eq:h_n_time} and
make use of the inequality $\Gamma_j\geq
\energy(0)N_*^{-n}=\energy(0)2^{-2\tilde\alpha n}$,
\[
\frac{\lambda_jh_n}{2}
=N_*2^{n\alpha+2\alpha}\sqrt{\Gamma_j}\frac{\sqrt{2}\energy(0)^{3/2}}{2^{(\alpha-\tilde\alpha) n}\delta_n^2}
\geq \frac{\energy(0)^2}{\delta_n^2}\sqrt{2}
\]
By the hypothesis that $\sum_{|j|=n}m_j^2>\delta_n$ and
Lemma~\ref{lem:stime_sup_inf}, we know that $\delta_n<\energy(0)$ we
get $1-e^{-\lambda_jh_n/2}>\frac12$.
We obtain
\begin{equation}\label{eq:25bis}
\energy(0)
> \sum_{|j|=n}m_j^4\frac{h_n}{2\lambda_j}\sum_{k\in\O j}c_k^2
= \frac{\sqrt{2}\energy(0)^{3/2}}{2^{-\tilde\alpha n}\delta_n^2}\sum_{|j|=n}\frac{m_j^4}{\sqrt{\Gamma_j}}
\end{equation}
Now we can use Cauchy-Schwarz and the AM-QM inequalities to get
\[
\sum_{|j|=n}\frac{m_j^4}{\sqrt{\Gamma_j}}
\geq\frac{\left(\sum_{|j|=n}m_j^2\right)^2}{\sum_{|j|=n}\sqrt{\Gamma_j}}
\geq\frac{\left(\sum_{|j|=n}m_j^2\right)^2}{\sqrt{N_*^n\sum_{|j|=n}\Gamma_j}}
\]
again by the hypothesis that $\sum_{|j|=n}m_j^2>\delta_n$ and thanks
to $\sum_{|j|=n}\Gamma_j\leq2\energy(0)$,
\[
\sum_{|j|=n}\frac{m_j^4}{\sqrt{\Gamma_j}}
>\frac{\delta_n^2}{\sqrt{2\energy(0)}2^{\tilde\alpha n}}
\]
so that the right-hand side of~\eqref{eq:25bis} becomes larger than
$\energy(0)$, which is impossible.

We turn to the second part.  Let $M\geq 0$ and define the following
sequence $(t_n)_{n\geq M-1}$ by $t_{M-1}=0$ and
$t_n=t_{n-1}+h_n$. By~\eqref{eq:tesi_lemma1} with $t=t_{n-1}$ we get
\begin{equation*}
  \energy_n(t_n)-\energy_{n-1}(t_{n-1})\leq \delta_n.
\end{equation*}
We sum for $n$ from $M$ to $N$, yielding
\[
\energy_N(t_N)-\energy_{M-1}(0)\leq \sum_{n=M}^N\delta_n
\]
which, due to monotonicity of $\energy_N$, yields
\[
\energy_N\biggl(\sum_{n=M}^\infty h_n\biggr)
\leq \energy_N(t_N)\leq \energy_{M-1}(0)+\sum_{n=M}^N\delta_n.
\]
Now we let $N$ go to infinity to get the thesis.
\end{proof}

\begin{rmk}
It is easy to prove this result also if relaxing the condition on the
number of children from constant number to $1\leq\sharp\O j\leq
N_*$. One has to change slightly the definition of $h_n$, which
becomes
\[
h_n
=\frac{\energy(0)^{3/2}}{\delta_n^2}2^{-(\alpha-\tilde\alpha)n+2\tilde\alpha+3/2}.
\]
\end{rmk}

\begin{thm}\label{thm:limit-energy}
Assume that $\alpha>\tilde\alpha$, where $2^{2\tilde\alpha}=N_*=
\sharp\O j$ is the constant number of children for every node.
Then for every $\epsilon>0$ and $\eta>0$ there exists some $T>0$ such
that for all positive $l^2$ solution of
\eqref{eq:System_inviscid_unforced} with initial energy
$\energy(0)\leq\eta$ one has $\energy(T)\leq\epsilon$. In particular
\[
\lim_{t\to\infty}\energy(t)=0
\]
i.e.\ there is anomalous dissipation.
\end{thm}

\begin{proof}
Given $\epsilon>0$ let us take a sequence of positive numbers
$(\delta_n)_{n\geq0}$ such that
\[
\sum_{n=o}^\infty \delta_n
=\epsilon\qquad \mathrm{and}\qquad \sum_{n=o}^\infty\frac{1}{2^{(\alpha-\tilde\alpha)n}\delta_n^2}<+\infty
\]
This is possible, for example, taking
$\delta_n=\epsilon(1-2^{-(\alpha-\tilde\alpha)/3})2^{-(\alpha-\tilde\alpha)n/3}$.
Now Lemma~\ref{lem:en_dissip} applies, so by the definition of $h_n$
given in~\eqref{eq:h_n_time}
\[
h_n
\leq\frac{2\sqrt{2}\eta^{3/2}}{2^{(\alpha-\tilde\alpha) n}\delta_n^2}
\qquad \text{and}\qquad
\sum_{n=0}^\infty h_n
\leq\frac{2\sqrt{2}\eta^{3/2}}{(1-2^{-(\alpha-\tilde\alpha) /3})^3}
=:T.
\]
Take $M=0$ in~\eqref{eq:tesi_lemma2} and by monotonicity of energy
$\energy(T)\leq\epsilon$.
\end{proof}



We are finally able to prove Theorem~\ref{thm:main_1}, which is a
consequence of Theorem~\ref{thm:limit-energy} with a rescaling
argument based on the fact that the non-linearity is homogeneous of
degree two.

\begin{proof}[Proof of Theorem~\ref{thm:main_1}]
By Theorem~\ref{thm:limit-energy} for every $0<\rho<1$ there exists
$\tau>0$ depending only on $\rho$ and $\energy(0)$, such that
$\energy(\tau)\leq\rho^2\energy(0)$. We will apply this bound to many
different solutions, all of which have energy at time zero not above
$\energy(0)$.

Let $\vartheta=1/\rho>1$. We can define the sequence
\begin{align*}
  &X^{(0)}=X\\
&X^{(n)}(t)=\vartheta X^{(n-1)}(\vartheta t +
  \tau)=\vartheta^n X\biggl(\vartheta^n t +
  \frac{\vartheta^n-1}{\vartheta -1}\tau\biggr),\qquad n\geq1
\end{align*}
It is immediate to verify that all of these satisfy the system of
equations~\eqref{eq:System_inviscid_unforced}, but with possibly
different initial conditions. We have
\[
\sum_{j\in J}\bigl(X_j^{(n)}(0)\bigr)^2
=\vartheta^2 \sum_{j\in J}\bigl(X_j^{(n-1)}(\tau)\bigr)^2.
\]
Recalling the definition of $\tau$, the above equation allows to prove
by induction on $n$ that for all $n\geq0$ one has $\sum_{j\in
  J}\bigl(X_j^{(n)}(0)\bigr)^2 \leq \energy(0)$. For all $n\geq0$, let
\[
t_n=\frac{\vartheta^n-1}{\vartheta -1}\tau.
\]
Then by the definition of $X^{(n)}$, we have proved $\energy(t_n)^2
\leq \vartheta^{-2n}\energy\left(0\right)$.  Since $\vartheta>1$,
$t_n\uparrow\infty$, hence given $t>0$ there is $n$ such that $t_n\leq
t < t_{n+1}$. That means we have by monotonicity
\[
\energy(t)
\leq \vartheta^{-2n}\energy(0)
\qquad\text{and}\qquad
\frac1{t_{n+1}^2}
<\frac1{t^2}
\]
finally, by definition $t_{n+1}<
\frac{\vartheta^{n+1}}{\vartheta-1}\tau=\vartheta^n\frac\tau{1-\rho}$,
so for $C=\energy(0)\bigl(\frac\tau{1-\rho}\bigr)^2$ we get
\[
\energy(t)
\leq\vartheta^{-2n}\energy\left(0\right)
<\frac C{t_{n+1}^2}
<\frac C{t^2}\qedhere
\]
\end{proof}

\subsection{Self-similar solutions}
\label{subsec:self_similar}
We devote the end of this section to prove the existence of
self-similar solutions. We call \emph{self-similar} any solution $X$
of system~\eqref{eq:System_inviscid_unforced} of the form
$X_j(t)=a_j\varphi(t)$, for all $j$ and all $t\geq0$. By substituting
this formula inside~\eqref{eq:System_inviscid_unforced} it is easy to
show that any such solution must be of the form
\[
X_j(t)=\frac{a_j}{t-t_0}
\]
for some $t_0<0$.  The condition on the coefficients $a_j$ is much
more complicated
\[
\left\{
\begin{aligned}
&a_{\bar0}=0\\
&a_j + c_j a_{\bar\jmath}^2 = \sum_{k\in\O j}c_ka_ja_k,  \qquad\forall j\in J
\end{aligned}
\right.
\]
so we base instead our argument upon~\cite{MR2746670}, where it is
proven existence and some kind of uniqueness of self-similar solution. 
We obtain the following statement.
\begin{prop}
Given $t_0<0$ there exists at least one self-similar positive $l^2$
solution of~\eqref{eq:System_inviscid_unforced} with $a_0>0$.
\end{prop}
\begin{proof}
We use Theorem~10 in~\cite{MR2746670} which, translated in the
notation of this paper, states that there exists a unique sequence of
non-negative real numbers $(b_n)_{n\geq0}$ such that $b_0>0$ and
$Y_n:=\frac{b_n}{t-t_0}$ is a positive $l^2$ solution of the unforced
inviscid classic dyadic~\eqref{eq:system_classic_dyadic_general}.
Thanks to Proposition~\ref{prop:class_to_tree} this solution may be
lifted to a solution of the inviscid tree
dyadic~\eqref{eq:system_tree_dyadic_general} with the required
features.
\end{proof}
\begin{rmk}
For the tree dyadic model self-similar solutions are many. In the
standard dyadic case studied in~\cite{MR2746670} it is shown that
given $t_0<0$ and $n_0\geq1$ there is only one $l^2$ self-similar
solution such that $n_0$ is the index of the first non-zero
coefficient. If $n_0>1$, this solution can be lifted on the tree to a
self-similar solution which is zero on the first $n_0-1$
generations. We can then define a new self-similar solution which is
equal to this one on one of the subtrees starting at generation $n_0$
and zero everywhere else. Finally, we can combine many of these
solutions, even with different $n_0$, as long as $t_0$ is the same for
all and their subtrees do not overlap.
\end{rmk}

\section{Stationary solutions}
\label{sec:stat-sol}
In this section we will study the stationary solutions for both the
classic dyadic model~\eqref{eq:system_diadico} and the tree dyadic
one~\eqref{eq:system_albero}. We will in particular restrict ourselves
to study positive $l^2$ solutions which are time independent.
Proposition~\ref{prop:class_to_tree} allows us to link the two models,
in that for any solution of the classic dyadic model one can build a
solution of the tree dyadic model. Thus is it enough to prove
existence for the classic dyadic and uniqueness for the tree dyadic.

One purpose of this section is to prove the existence and uniqueness
of stationary solution on the tree dyadic model and extend existence
and uniqueness results given in~\cite{CheFriPav2007}
and~\cite{MR2522972} for the dyadic model.  In~\cite{CheFriPav2007} it
is proven that the inviscid dyadic model with $\beta=\frac{5}{2}$ has
a unique stationary solution, while in the companion
paper~\cite{CheFriPav2010} it is proven that such a solution is a
global attractor.  The viscous dyadic model is studied
in~\cite{MR2522972}, where it is proven that for
$\beta\in(\frac{3}{2},\frac{5}{2}]$ the stationary solution is unique
and is a global attractor.  In~\cite{BarMorRom} it is proven that for
the viscous case it is possible, dropping the $Y_n\geq0$ condition, to
explicitly provide examples of non-uniqueness of the stationary
solution.  In this paper we prove the existence and uniqueness of
stationary solutions in $l^2$ for every positive value of the $\beta$
and $\gamma$ parameters both in viscous and inviscid dyadic models.
This will provide a corresponding result of existence and uniqueness
for $\alpha>\tilde\alpha$ and $\gamma>0$ in the tree dyadic model.
Furthermore in the inviscid case we will explicitly provide those
solutions (Proposition~\ref{prop:sol_sta}), while in the viscous case
we'll prove that the stationary solutions are regular if and only if
$N_*$ is big enough, $N_*\geq 2^{2\alpha-3\gamma}$ or the forcing term
$f$ is small.  For $f=0$ the unique (non-negative) stationary solution
is trivially the null one, so in this section we assume $f>0$.

\subsection{Stationary solutions in the inviscid case: existence.}

In the inviscid case, the differential equation is very simple, so it
is easy to find stationary solutions in the class of exponential
functions. One immediately finds the following result.

\begin{prop}\label{prop:sol_sta}
Consider the tree dyadic model~\eqref{eq:system_albero} and the
classic dyadic model~\eqref{eq:system_diadico}, both inviscid
($\nu=0$). Let $2^{2\tilde\alpha}=N_*=\sharp\O j$ be constant for all
$j\in J$. Then:
\begin{enumerate} 
\item the sequence of constant functions $Y_n(t):=f\cdot
  2^{-\frac{\beta}{3}(n+1)}$ is a positive $l^2$ solution of
  the system~\eqref{eq:system_diadico}.
\item the family of constant functions $X_j(t):=f\cdot
  2^{-\frac{|j|+1}{3}(2\tilde\alpha+\alpha)}$ for $j\in J$ is a
  positive componentwise solution of system~\eqref{eq:system_albero};
  it is also an $l^2$ solution iff $\alpha>\tilde\alpha$;
\end{enumerate}
\end{prop}
\begin{proof}
A direct computation shows that $X$ and $Y$ are componentwise
solutions. To show that $Y$ is $l^2$ observe that, since $\beta>0$,
$\|Y\|<\infty$. To check whether $X$ is $l^2$ compute the energy by
generations; we have for $n\geq0$,
\[
\energy_n-\energy_{n-1}
=\sum_{|j|=n}X_j^2
=2^{2\tilde\alpha n}f^2\cdot2^{-\frac{n+1}{3}(4\tilde\alpha+2\alpha)}
=C\cdot 2^{\frac{2}{3}\left(\tilde\alpha-\alpha\right)n}
\]
with $C$ not depending on $n$. Hence $X$ is $l^2$ if and only if
$\alpha-\tilde\alpha>0$.
\end{proof}

\subsection{Stationary solutions in the viscous case: existence.}
In the viscous case, the recurrence relation coming from the
definition of stationary solution is more complex, and has no
solutions in the class of exponential functions. Anyway, by careful
control of the recurrence behavior, we are able to prove that a
stationary solution exists, and also to distinguish if it is
conservative or has anomalous dissipation.

\begin{defi}
We say that a stationary positive $l^2$ solution $X$ is regular if for
all $h\in\R$
\begin{equation}\label{eq:def_stat_regular}
\sum_{j\in J}[2^{h|j|}X_j]^2<\infty
\end{equation}
\end{defi}

\begin{thm}\label{thm:esi_sol_sta_dia_vis}
There exists a stationary positive $l^2$
solution of the classic dyadic model~\eqref{eq:system_diadico} when
$\nu>0$.
\end{thm}

\begin{thm}\label{thm:reg_sol_sta_dia_vis}
Consider any stationary positive $l^2$ solution of the classic dyadic
model~\eqref{eq:system_diadico} with $\nu>0$.
\begin{enumerate}
\item If $3\gamma\geq2\beta$ then it is regular and conservative.
\item If $3\gamma<2\beta$ then there exists some $C>0$ such that if
  $f>C$ the stationary solution is not regular and there is anomalous
  dissipation.
\end{enumerate}
\end{thm}

Before we go into the proofs of these theorems, let us introduce a
useful change of variables, that will come handy in both proofs. If
$Y$ is a stationary solution of~\eqref{eq:system_diadico} then, for
every $n\geq0$, we have
\[
-\nu 2^{\gamma n} Y_n+ 2^{\beta n} Y_{n-1}^2-2^{\beta n+\beta}Y_nY_{n+1}=0
\]
This equation can be made into a recurrence, and the change of
variables that best simplifies its form is
\begin{equation}\label{eq:change_var}
  Z_n:=\nu^{-1}2^{\frac{\beta}{3}(n+2)} Y_n. 
\end{equation}
Since the stationary solution in the inviscid case decreases like
$2^{-\frac{\beta}{3}n}$, the exponent's rate $\frac{\beta}{3}n$ is in
some sense expected.  The system of differential equations for $Z$
becomes
\begin{equation}\label{eq:system_Z}
\left\{\begin{aligned}
Z_{-1} &= \nu^{-1}2^{\frac{\beta}{3}} f=:g\\ 
Z_{n+1}&=\frac{Z_{n-1}^2}{Z_n}- 2^{(\gamma-\frac{2}{3}\beta)n} 
\qquad \qquad \forall n\geq 0 
\end{aligned}\right.
\end{equation}

\begin{proof}[Proof of theorem~\ref{thm:esi_sol_sta_dia_vis}]
Let us consider the change of variable \eqref{eq:change_var}, we have
to show that the system~\eqref{eq:system_Z} has a positive solution
for which $Y$ is $l^2$.  System~\eqref{eq:system_Z} gives a recursion
which, given $Z_{-1}=g$ and $Z_0$ allows to construct the sequence
$(Z_n)_{n\geq-1}$ in a unique way. Any such sequence will give a
stationary componentwise solution. What we want to prove is that there
is some value of $Z_0$ such that this turn out to be a positive $l^2$
solution. Let we exploit the dependence from $Z_0$ by defining a
sequence of real functions
\begin{equation}\label{eq:def_ind_Zn}
\begin{aligned}
Z_{-1}(a)&=g\\
Z_0(a)&=a\\
Z_{n+1}(a)&=\tfrac{Z_{n-1}^2(a)}{Z_n(a)}- 2^{(\gamma-\frac{2}{3}\beta)n} ,\qquad n\geq0
\end{aligned}
\end{equation}
Now we construct a descending sequence of open real intervals
$(I_n)_{n\geq0}$ such that $(0,\infty)=I_0\supset I_1\supset
I_2\supset\dots$ and such that $Z_n$ is continuous and bijective from
$I_n$ to $(0,\infty)$, with $Z_n$ strictly increasing for even $n$ and
strictly decreasing for odd $n$.

Let $I_0=(0,\infty)$. $Z_0(a)$ is monotone increasing, continuous and
bijective from $I_0$ to $(0,\infty)$.

By~\eqref{eq:def_ind_Zn} we have that
$Z_1(a)=g/a^2-2^{(\gamma-\frac{2}{3}\beta)}$ is monotone decreasing,
continuous and bijective from $I_0$ to
$(-2^{(\gamma-\frac{2}{3}\beta)},\infty)$ so there exists a limited
interval $(b_1,c_1):=I_1\subset I_0$ such that $Z_1(a)$ is
monotone decreasing, continuous and bijective from $I_1$ to
$(0,\infty)$.

Now suppose we already proved for $m\leq n$ that $Z_m(a)$ is
continuous and bijective from $I_m$ to $(0,\infty)$, with $Z_m$
strictly increasing for even $m$ and strictly decreasing for odd $m$.

Suppose that $n$ is odd (resp.\ even). Then by~\eqref{eq:def_ind_Zn}
$Z_{n+1}(a)$ is monotone increasing (resp.\ decreasing), continuous
and bijective from $I_n$ to $(-2^{(\gamma-\frac{2}{3}\beta)n},\infty)$
so there exists an interval $(b_{n+1},c_{n+1}):=I_{n+1}\subset
I_n$ such that $Z_{n+1}(a)$ is monotone increasing
(resp.\ decreasing), continuous and bijective from $I_{n+1}$ to
$(0,\infty)$.

Observe moreover that the borders of these intervals are not
definitively constant, since for all $n$, $b_{n+2}\neq b_n$ and
$c_{n+2}\neq c_n$. Hence if we define $b=\lim_n b_n$ and $c=\lim_n
c_n$, it is clear that for all $n$, $b_n<b\leq c<c_n$, that is the
closed interval (possibly degenerate) $[b,c]$ is contained in every
$I_n$.

Now we choose any $\bar a\in[b,c]$ and we know that the sequence
$Z_n(\bar a)$ is strictly positive. We are left to prove that it is
also $l^2$. To this end let $Y_n$ be any stationary, positive
componentwise solution. Let $\energy_n=\sum_{k=0}^n Y_k^2$ in analogy
with the definition for the tree model. We compute the derivative
\[
0
=\frac d{dt}\energy_n(t)
=-\nu\sum_{k\leq n}l_kY_k^2+f^2Y_0-k_{n+1}Y_n^2Y_{n+1},
\]
hence, since $l_k\geq1$, $\energy_n\leq\sum_{k\leq
  n}l_kY_k^2\leq\nu^{-1}f^2Y_0$ for all $n$.
\end{proof}

\begin{proof}[Proof of theorem~\ref{thm:reg_sol_sta_dia_vis}]
Let us consider again system~\eqref{eq:system_Z} and let
$\mu:=\gamma-\frac{2}{3}\beta$. If $\mu>0$ the corrective term goes to
infinity, while if $\mu<0$ it goes to zero, so we expect two different
behaviors in the two cases. We'll show that in the first case $Z_n$
goes to zero super-exponentially for $n\to \infty$, while in the second
one $Z_n\downarrow z$ and $z>0$ if $g$ is large enough.

\noindent
\emph{Case $\mu:=\gamma-\frac{2}{3}\beta\geq0$.}  
From~\eqref{eq:system_Z} we get
\[
2^{\mu n}Z_n^2=Z_{n-1}^2Z_{n}-Z_n^2Z_{n+1}
\]
Sum over $n$ to get
\begin{equation}\label{eq:energy_flux_stationary}
\sum_{k\leq n}2^{\mu k}Z_k^2=g^2Z_0-Z_n^2Z_{n+1}
\end{equation}
Since $\mu\geq0$, by positivity of $Z$, we have
\begin{equation} \label{eq:limZ_n=0}
 \lim_{n\to \infty} Z_n=0.
\end{equation}
From~\eqref{eq:system_Z} and $Z_{n+1}>0$ we get $Z_n<Z_{n-1}^2$
and since by~\eqref{eq:limZ_n=0} $Z_{\bar n}=:\lambda<1$ for some
$\bar n$, by iterating the above equation we get for all $m\geq0$
\[
Z_{\bar{n}+m}\leq\lambda^{2^m}
\]
that is to say that $Z_n$ goes to zero for $n$ going to infinity like
the exponential of an exponential, so for every $s>0$ we have
\[
\sum_n \left(2^{sn} Z_n  \right)^2<+\infty \qquad \text{and} \qquad
\sum_n \left(2^{sn} Y_n  \right)^2<+\infty.
\]
It is now clear that $\lim_nk_{n+1}Y_n^2Y_{n+1}=0$, so $Y$ is
conservative by Definition~\ref{def:conservative}.

\emph{Case $\mu:=\gamma-\frac{2}{3}\beta<0$.}  The first step is to
prove that $Z_n$ is non-increasing in $n$. Suppose by contradiction
that for some $n$ we have $\frac{Z_n}{Z_{n-1}}=\lambda>1$, then we
claim that $\frac{Z_{n+2}}{Z_{n+1}}>\lambda^4>1$ and hence by
induction $\frac{Z_{n+2m}}{Z_{n+2m-1}}>\lambda^{4^m}$.
By~\eqref{eq:system_Z} for all $k\geq0$
\[
Z_{k+1}<\frac{Z^2_{k-1}}{Z_k}
=\frac{Z_{k-1}}{Z_k}Z_{k-1}
\]
This can be used iteratively together with the claim to show that
\begin{multline*}
Z_{n+2m+1}
=\frac{Z^2_{n+2m-1}}{Z_{n+2m}}-2^{\mu(n+2m)}\\
<\frac{Z_{n+2m-1}}{Z_{n+2m}}\frac{Z_{n+2m-3}}{Z_{n+2m-2}}\dots\frac{Z_{n-1}}{Z_{n}}Z_{n-1}-2^{\mu(n+2m)}\\
<Z_{n-1}\lambda^{-4^m}-2^{\mu(n+2m)}
\end{multline*}
so we get a contradiction because $Z_{n+2m+1}<0$ for some $m$.

We prove the claim. Let $x=\frac{2^{\mu n}Z_n}{Z^2_{n-1}}=\frac{2^{\mu
    n}\lambda^2}{Z_n}$.  Observe that
\begin{equation}\label{eq:z_n_plus_one}
 Z_{n+1}
=\frac{Z^2_{n-1}}{Z_n}- 2^{\mu n}
=\frac{2^{\mu n}}x(1-x)
\end{equation}
We divide by $Z_n$ (and we notice that $x<1$),
\begin{equation}\label{eq:ratio_z_n_plus_one}
\frac{Z_{n+1}}{Z_n}
=\lambda^{-2}(1-x)
\end{equation}
Now
\[
Z_{n+2}
=\frac{Z^2_{n}}{Z_{n+1}}- 2^{\mu (n+1)}
>\frac{Z^2_{n}}{Z_{n+1}}- 2^{\mu n}
\]
so dividing by $Z_{n+1}$ and
substituting~\eqref{eq:ratio_z_n_plus_one} and~\eqref{eq:z_n_plus_one}, we get
\[
\frac{Z_{n+2}}{Z_{n+1}}
>\lambda^4(1-x)^{-2}- \frac{2^{\mu n}}{Z_{n+1}}
>\frac{\lambda^4}{1-x}-\frac x{1-x}
\]
Since $\lambda>1>x>0$, it is now clear that
$\frac{\lambda^4-x}{1-x}>\lambda^4$.  So we have proven the claim and
showed that $\{Z_n\}_{n\geq0}$ is non-increasing in $n$.

The last step is to show that for $g$ large enough $Z_n\downarrow
z>0$. By rearranging~\eqref{eq:energy_flux_stationary} and recalling
what we proved above,
\[
Z_n^3
\geq Z_n^2Z_{n+1}
=g^2Z_0-\sum_{k=0}^n2^{\mu k}Z_k^2
\geq g^2Z_0-gZ_0\sum_{k=0}^n2^{\mu k}
> gZ_0\Bigl(g-\frac1{1-2^{\mu}}\Bigr)
\]
so if $g>\frac1{1-2^{\mu}}$ then $Z_n$ converges to a strictly
positive constant $z$.

To prove anomalous dissipation we compute the limit
\[
\lim_{n\to\infty}k_{n+1}Y_n^2Y_{n+1}
=\lim_{n\to\infty}2^{\beta n+\beta}\nu^32^{-\beta n-7\beta/3}Z_n^2Z_{n+1}
=2^{-4\beta/3}\nu^3z^3>0
\]
So by Definition~\ref{def:conservative} there is anomalous
dissipation.
\end{proof}

\subsection{Stationary solutions in the inviscid and viscous case: uniqueness}
We prove uniqueness in the class of stationary positive $l^2$
solutions for the tree dyadic model. The result also holds for the
classic dyadic, because it is a particular case of the former, or by
virtue of the lifting Proposition~\ref{prop:class_to_tree}.
\begin{thm}
\label{thm:esi_unic_sol_inv}
Consider the tree dyadic model~\eqref{eq:system_tree_dyadic_general}
and assume that $\alpha>\tilde\alpha$, where $2^{2\tilde\alpha}=N_*=
\sharp\O j$ is the constant number of children for every node.  Then
there exists a unique stationary positive $l^2$ solution.
\end{thm}
\begin{proof}
Existence is a consequence of Proposition~\ref{prop:sol_sta} in the
inviscid case ($\nu = 0$) and Proposition~\ref{prop:class_to_tree} and
Theorem~\ref{thm:esi_sol_sta_dia_vis} in the viscous case.
 
To prove uniqueness we apply a change of variables similar
to~\eqref{eq:change_var} 
\begin{equation}\label{eq:chng_var_uniq}
Z_j:=2^{\frac{(2+|j|)\alpha}{3}}X_j,
\qquad\forall j\in J
\end{equation}
Then from~\eqref{eq:system_tree_dyadic_general} we have
\begin{equation}\label{eq:tree_visc_risc}
\frac{d}{dt} Z_j = 
- \nu 2^{\gamma|j|} Z_j+
2^{\frac{2}{3}\alpha|j|} Z_{\bar{\jmath}}^2
-\sum_{k\in\O j}2^{\frac{2}{3}\alpha|j|}Z_jZ_k
\end{equation}
so if $X$ is a stationary solution, $Z$ must satisfy
\begin{equation}\label{eq:tree_visc_risc_sta}
\left\{\begin{aligned}
Z_{\bar0}&=f\cdot 2^{\alpha/3}\\
\sum_{k\in\O j}Z_k&=
 \frac{Z_{\bar{\jmath}}^2}{Z_j}
 - \nu 2^{(\gamma-\frac{2}{3}\alpha)|j|} .
\end{aligned}\right.
\end{equation}
Moreover observe that the condition $X\in l^2$ is equivalent to
\begin{equation}\label{eq:Z_regularity}
 \sum_{j\in J} \bigl(2^{-\frac{\alpha}{3}|j|}Z_j \bigr)^2<\infty
\end{equation}


Assume by contradiction that there are two different stationary
solutions of~\eqref{eq:tree_visc_risc_sta} which we denote by
$W=\{W_j\}_{j\in J}$ and $Z=\{Z_j\}_{j\in J}$. Let $n$ be the smallest
integer such that there exist $j_1\in J$ with $|j_1|=n$ and
$W_{j_1}\neq Z_{j_1}$. Without loss of generality we can take
$\frac{W_{j_1}}{Z_{j_1}}=:\lambda>1$.

Let $j_0=k_0=\bar{\jmath_1}$ and $k_1=j_1$. Extend these to two
sequences of indices $(j_m)_{m\geq0}$ and $(k_m)_{m\geq0}$ with
$j_m\in\O{j_{m-1}}$ and $k_m\in\O{k_{m-1}}$, picking alternatively
among those that maximize or minimize $W_{j_m}$ and $Z_{k_m}$.

More precisely for $m\geq2$ choose $j_m\in\O{j_{m-1}}$ and
$k_m\in\O{k_{m-1}}$ in such a way that if $m$ is even
\begin{align*}
W_{j_m}&=\min\{W_i:i\in\O{j_{m-1}}\} &
Z_{k_m}&=\max\{Z_i:i\in\O{k_{m-1}}\}
\end{align*}
and if $m$ is odd
\begin{align*}
W_{j_m}&=\max\{W_i:i\in\O{j_{m-1}}\} &
Z_{k_m}&=\min\{Z_i:i\in\O{k_{m-1}}\}
\end{align*}
The idea supporting the definition of these sequences is to choose the
indices so that
\[
W_{j_1}<Z_{k_1},\qquad W_{j_2}>Z_{k_2},\qquad W_{j_3}<Z_{k_3},\qquad
\ldots
\]
We will now prove that, with our construction, those inequalities hold
and, moreover, the ratio between $W_m$ and $Z_m$ grows according to
\begin{align}
  \label{eq:ratio1}
  \frac{Z_{k_m}}{W_{j_m}}
\geq \frac{W_{j_{m-1}}}{Z_{k_{m-1}}}\cdot \frac{Z^2_{k_{m-2}}}{W^2_{j_{m-2}}}
>\lambda^{2^{m-2}}
\qquad \qquad 
\forall\, m\geq 2\ \text{even}\\
\label{eq:ratio2}
\frac{W_{j_m}}{Z_{k_m}}
\geq \frac{Z_{k_{m-1}}}{W_{j_{m-1}}}\cdot \frac{W^2_{j_{m-2}}}{Z^2_{k_{m-2}}}
>\lambda^{2^{m-2}}
\qquad \qquad 
\forall\, m\geq 3 \ \text{odd}.
\end{align}
We prove inequalities~\eqref{eq:ratio1} and~\eqref{eq:ratio2} by
induction on $m\geq2$.  First note that for $m=0$ and $m=1$,
\begin{equation}\label{eq:ratio_Z_W_basecase}
\frac{Z_{k_0}}{W_{j_0}}=1
\qquad\text{and}\qquad
\frac{W_{j_1}}{Z_{k_1}}=\lambda
\end{equation}
Now we proceed by induction. Let $m\geq2$ even.  By the definition of
$j_m$, $k_m$ and by~\eqref{eq:tree_visc_risc_sta} we get
\begin{gather}  
  W_{j_m}
=\min_{i\in\O{j_{m-1}}}W_i
\leq N_*^{-1}\sum_{i\in\O{j_{m-1}}}W_i
= N_*^{-1}\biggl[\frac{W^2_{j_{m-2}}}{W_{j_{m-1}}}
-\nu 2^{(\gamma-\frac{2}{3}\alpha)(n+m-2)}\biggr]\\
  Z_{k_m}
=\max_{i\in\O{k_{m-1}}}Z_i
\geq N_*^{-1}\sum_{i\in\O{k_{m-1}}}Z_i
= N_*^{-1}\biggl[\frac{Z^2_{k_{m-2}}}{Z_{k_{m-1}}}
-\nu 2^{(\gamma-\frac{2}{3}\alpha)(n+m-2)}\biggr]\label{eq:Zkm_geq}
\end{gather}
By~\eqref{eq:ratio_Z_W_basecase} when $m=2$ or by inductive
hypothesis~\eqref{eq:ratio1} and~\eqref{eq:ratio2} when $m\geq4$,
\[
\frac{Z^2_{k_{m-2}}}{Z_{k_{m-1}}}\Biggl/\frac{W^2_{j_{m-2}}}{W_{j_{m-1}}}
=\frac{Z^2_{k_{m-2}}}{W^2_{j_{m-2}}}\frac{W_{j_{m-1}}}{Z_{k_{m-1}}}
\geq\begin{cases}
\lambda & m=2 \\
(\lambda^{2^{m-4}})^2\lambda^{2^{m-3}}=\lambda^{2^{m-2}} & m\geq4
\end{cases}
\]
so in particular the ratio is above 1 and, since for every
$a>b>c\geq0$ we have $\frac{a-c}{b-c}\geq\frac{a}{b}$, for $m\geq2$
even
\[
\frac{Z_{k_m}}{W_{j_m}}
=\frac{\frac{Z^2_{k_{m-2}}}{Z_{k_{m-1}}}-\nu 2^{(\gamma-\frac{2}{3}\alpha)(n+m-2)}}
{\frac{W^2_{j_{m-2}}}{W_{j_{m-1}}}-\nu 2^{(\gamma-\frac{2}{3}\alpha)(n+m-2)}}
\geq\frac{Z^2_{k_{m-2}}}{Z_{k_{m-1}}}\Biggl/\frac{W^2_{j_{m-2}}}{W_{j_{m-1}}}
\geq\lambda^{2^{m-2}}
\]
This concludes the inductive step for $m$ even; for $m$ odd the
reasoning is analogous.  We now want to use inequalities
\eqref{eq:ratio1} and \eqref{eq:ratio2} to get a contradiction.  We
will consider separately the cases $\nu>0$ and $\nu=0$.

\noindent\emph{Case $\nu>0$.}  Let $m$ be even; by~\eqref{eq:ratio1}
\[
\frac{Z^2_{k_{m-2}}}{Z_{k_{m-1}}}
>\lambda^{2^{m-2}}\frac{W^2_{j_{m-2}}}{W_{j_{m-1}}}
\]
applying~\eqref{eq:tree_visc_risc_sta} to $W_{j_{m-1}}$ we have
\[
\frac{W^2_{j_{m-2}}}{W_{j_{m-1}}}
\geq\nu2^{(\gamma-\frac{2}{3}\alpha)(n+m-2)}
\]
so from~\eqref{eq:Zkm_geq}, putting everything together, we get
\[
Z_{k_m}
\geq N_*^{-1}\biggl[\frac{Z^2_{k_{m-2}}}{Z_{k_{m-1}}}-\nu2^{(\gamma-\frac{2}{3}\alpha)(n+m-2)}\biggr]
\geq N_*^{-1}\nu2^{(\gamma-\frac{2}{3}\alpha)(n+m-2)}\bigl(\lambda^{2^{m-2}}-1\bigr)
\]
For $m$ even going to infinity we have obviously that $Z_{k_m}$ grows
as the exponential of an exponential, which is in contradiction with
\eqref{eq:Z_regularity}.

\noindent\emph{Case $\nu=0$.}  If $\nu=0$ we already know one
explicit stationary solution, by Proposition~\ref{prop:sol_sta},
namely $X_j=f 2^{-\frac{|j|+1}{3}(2\tilde\alpha+\alpha)}$.  By the
usual change of variables~\eqref{eq:chng_var_uniq}
$V_j=f2^{\frac23(\alpha-\tilde\alpha(|j|-1))}$ is a solution
of~\eqref{eq:tree_visc_risc_sta} satisfying the regularity
condition~\eqref{eq:Z_regularity}. Without loss of generality we can
suppose that $W_j=V_j$ or $Z_j=V_j$. In the first case, for $m$ even
\[
Z_{k_m}>W_{j_m}\lambda^{2^{m-2}}
=f2^{\frac23(\alpha-\tilde\alpha(n+m-2)}\lambda^{2^{m-2}}
\] 
in the second case for $m$ odd
\[
W_{j_m}>Z_{k_m}\lambda^{2^{m-2}}
=f2^{\frac23(\alpha-\tilde\alpha(n+m-2)}\lambda^{2^{m-2}}
\] 
In both cases the right-hand side grows super-exponentially as
$m\to\infty$ and this is in contradiction with \eqref{eq:Z_regularity}.
\end{proof}

\begin{proof}[Proof of Theorem~\ref{thm:main_3}]
Existence and uniqueness are given by
Theorem~\ref{thm:esi_unic_sol_inv}.

If $\nu=0$ the solution is identified by
Proposition~\ref{prop:sol_sta}. If $\nu>0$, by uniqueness, the
solution is the lift of the stationary solution of the classic dyadic
with $\beta=\alpha-\tilde\alpha$, as per
Proposition~\ref{prop:class_to_tree}. Then the two regimes are proven
in Theorem~\ref{thm:reg_sol_sta_dia_vis}.
\end{proof}

{\small
\bibliographystyle{habbrv}
\bibliography{bibdyadic}
}
\end{document}